\documentclass[a4paper,leqno]{amsart}
        \title{Algebraic K-theory of completed Kac-Moody groups}
       \author{Bartels, A.}
       \address{Mathematical Institute\\University M\"unster\\ Einsteinstr.~62, D-48149 M\"unster} 
       \email{a.bartels@uni-muenster.de}
        \urladdr{http://www.math.uni-muenster.de/u/bartelsa} 
        \author{L\"uck, W.}
      \address{Mathematisches Institut der Universit\"at Bonn\\
                Endenicher Allee 60\\
                D-53115 Bonn}
        \email{wolfgang.lueck@him.uni-bonn.de}
      \urladdr{http://www.him.uni-bonn.de/lueck}
      \author{Witzel, S.}
       \address{Mathematical Institute\\ Giessen University\\Arndtstr. 2,D-35392 Giessen} 
       \email{Stefan.Witzel@math.uni-giessen.de}
       \urladdr{http://www.switzel.eu/}
         \date{\today}
          \keywords{Kac-Moody-groups, Hecke algebras, algebraic K-theory}
     \makeatletter
     \@namedef{subjclassname@2020}{%
  \textup{2020} Mathematics Subject Classification}
\makeatother
\subjclass[2020]{primary 20G44, secondary: 20C08,19D50}

  \usepackage[utf8]{inputenc}
  \usepackage{hyperref}
  \usepackage{comment}
  \usepackage{calc}
  \usepackage{enumerate,amssymb,bm}
  \usepackage[arrow,curve,matrix,tips,2cell]{xy}
    \SelectTips{eu}{10} \UseTips
    \UseAllTwocells
  \usepackage{tikz}
  \usetikzlibrary{decorations.pathreplacing}
  \usepackage{color}
  \usepackage{scalerel}
  \usepackage{braket}
  \usepackage{mathtools}
  \usepackage{bbm}
  \usepackage{enumitem}
  \usepackage{float}
 
  \DeclareMathAlphabet{\matheurm}{U}{eur}{m}{n}

  \newcommand{\IN}{\mathbb{N}}

  \newcommand{\IQ}{\mathbb{Q}}
  \newcommand{\IR}{\mathbb{R}}

  \newcommand{\cala}{\mathcal{A}}
  
  \newcommand{\calc}{\mathcal{C}}

  \newcommand{\calh}{\mathcal{H}}

  \newcommand{\calo}{\mathcal{O}}

  \newcommand{\bfG}{{\mathbf G}}

   \newcounter{commentcounter}
   \newcounter{wlcommentcounter}

  \newcommand{\ignore}[1]{}

  \theoremstyle{plain}
  \newtheorem{theorem}{Theorem}[section]
  \newtheorem{lemma}[theorem]{Lemma}
  \newtheorem{corollary}[theorem]{Corollary}
  \newtheorem{proposition}[theorem]{Proposition}
  \newtheorem{addendum}[theorem]{Addendum}
  
  \newtheorem*{question*}{Question}
  \newtheorem*{theorem*}{Theorem}
  \newtheorem*{corollary*}{Corollary}

  \theoremstyle{definition}

  \newtheorem*{definition*}{Definition}

  \theoremstyle{remark}
    
  \newtheorem*{remark*}{Remark}

  \makeatletter\let\c@equation=\c@theorem\makeatother

  \DeclareMathOperator{\Aut}{Aut}
  \DeclareMathOperator{\id}{id}
  \DeclareMathOperator{\FS}{FS}
  \DeclareMathOperator{\Kgroup}{K}
  \DeclareMathOperator{\Sub}{Sub}
  \DeclareMathOperator*{\colimunder}{colim}
  \newcommand{\COP}{{\calc\hspace{-1pt}\mathrm{op}}}
  \newcommand{\CVCYC}{{\calc\hspace{-1pt}\mathrm{vcy}}}
   
\begin{document}

  \begin{abstract}
    We extend results for the K-theory of Hecke algebras of reductive $p$-adic groups to completed Kac-Moody groups.   
  \end{abstract}
  
   \maketitle

  \section*{Introduction}
\refstepcounter{section}
  
  A td-group $G$ is a topological group where the unit admits a countable neighborhood basis consisting of compact open subgroups.
  For a ring $R$ containing $\IQ$, the Hecke algebra $\calh(G;R)$ is the algebra\footnote{The product is the convolution with respect to a $\IQ$-valued Haar measure $\mu$ on $G$; the resulting algebra is up to canonical isomorphism independent of the choice of $\mu$.} of compactly supported locally constant functions $G \to R$. 
  In~\cite{Bartels-Lueck(2023K-theory_red_p-adic_groups)} the algebraic K-theory of $\calh(G;R)$ is studied in the spirit of the Farrell--Jones conjecture. 
  If $R$ is uniformly regular\footnote{This means that $R$ is noetherian and that there is $N$ such that every finitely generated $R$-module admits a projective resolution of length at most $N$ by finitely generated modules; in particular fields are uniformly regular.} and $G$ is a reductive p-adic group, then~\cite[Cor.~1.8]{Bartels-Lueck(2023K-theory_red_p-adic_groups)} expresses the K-theory of $\calh(G;R)$ in terms of the K-theories of $\calh(U;R)$ where $U$ varies over the compact open subgroups of $G$.
  On the level of $\Kgroup_0$ this yields an isomorphism~\cite[Cor.~1.2]{Bartels-Lueck(2023K-theory_red_p-adic_groups)}
   \begin{equation} \label{eq:K0-assembly}
  \colimunder_{U \in \Sub_\COP(G)} \Kgroup_0 (\calh(U;R)) \xrightarrow{\cong} \Kgroup_0 (\calh(G;R)).
 \end{equation}
 These results are consequences of the $\CVCYC$-Farrell--Jones conjecture for $G$, see~\cite[Thm.~5.15]{Bartels-Lueck(2023K-theory_red_p-adic_groups)}.
 The proof relies on the action of $G$ on its extended Bruhat-Tits building; this is a ${\rm CAT}(0)$-space.
 More precisely, this action is used in~\cite[Thm.~1.2]{Bartels-Lueck(2023almost)} to construct certain almost equivariant maps and the arguments in~\cite[Thm.~5.15]{Bartels-Lueck(2023K-theory_red_p-adic_groups)} then work for all td-groups admitting such almost equivariant maps.

 In the case of discrete groups $\Gamma$ the Farrell--Jones Conjecture holds for all ${\rm CAT}(0)$-groups, i.e.\ groups $\Gamma$ admitting a cocompact isometric proper action on a locally compact\footnote{In~\cite{Bartels-Lueck(2012annals)} the assumptions are stated slightly differently. There are different possible definitions for proper actions, either topological ("$\Gamma \times X \to X \times X$, $(\gamma,x) \mapsto (\gamma x,x)$ is proper") or metrically ("for bounded set $B$ there are only finitely many $\gamma$ with $\gamma B \cap B \neq \emptyset$"). Adding locally compact here ensures that it does not matter which of the two definitions are used.}  finite dimensional\footnote{By dimension we will usually mean covering dimension.}  ${\rm CAT}(0)$-space~\cite{Bartels-Lueck(2012annals)}.
 Similarly, one might ask whether  the $\CVCYC$-Farrell--Jones conjecture~\cite[Conj.~5.12]{Bartels-Lueck(2023K-theory_red_p-adic_groups)} holds for td-groups admitting a cocompact isometric proper smooth\footnote{Smooth means that the isotropy subgroups of the action are open.} action on a locally compact finite dimen\-sional ${\rm CAT}(0)$-space. 
 The arguments from~\cite{Bartels-Lueck(2023K-theory_red_p-adic_groups), Bartels-Lueck(2023almost)} do not prove this, because~\cite{Bartels-Lueck(2023almost)} relies on an additional property~\cite[Assum.~2.7]{Bartels-Lueck(2023almost)} of the action of $G$ on the ${\rm CAT}(0)$-space.
 This property is verified in~\cite{Bartels-Lueck(2023almost)} for the extended Bruhat-Tits building using that it is a Euclidean building.
 The purpose of this note is to observe that this property also holds for other buildings.
 This yields the following main result as a direct consequence of Propositions~\ref{prop:FJ-axiomatic} and~\ref{prop:group-theoretic}; see Section~\ref{sec:buildings-and-groups} for a discussion of $BN$-pairs.
 
 \begin{theorem} \label{thm:FJ-for-BN}
    If a td-group  $G$ admits a BN-pair $(B,N)$ in which $B$ is compact and open, then $G$ satisfies the $\CVCYC$-Farrell--Jones conjecture. 
  \end{theorem}

  Theorem~\ref{thm:FJ-for-BN} is a direct consequence of
    Propositions~\ref{prop:FJ-axiomatic}  and~\ref{prop:group-theoretic}.
 In fact, the result can be slightly strengthened, see Addendum~\ref{add:group-theoretic}.
 
 \begin{corollary}
 	The K-theoretic statements from \cite{Bartels-Lueck(2023K-theory_red_p-adic_groups)}\footnote{These are~\cite[Cor.~1.2, Cor.~1.8, Cor.~1.18]{Bartels-Lueck(2023K-theory_red_p-adic_groups)}.} all generalize from reductive p-adic groups to td-groups  $G$ admitting a BN-pair $(B,N)$ in which $B$ is compact and open.
 	In particular, \eqref{eq:K0-assembly} is an isomorphism for such $G$.
 \end{corollary}
 
 \begin{proof}
 	As explained in the introduction to~\cite{Bartels-Lueck(2023K-theory_red_p-adic_groups)} these statements are consequences of the $\CVCYC$-Farrell--Jones conjecture for $G$.
 \end{proof}

 In Section~\ref{sec:buildings-and-groups} we recall examples of groups admitting such BN-pairs.
 These include completed Kac--Moody groups as explained Proposition~\ref{prop:kac-moody_new}.
 
 \subsection*{Acknowledgments} 
   We thank Dorian Chanfi and Bernhard Mühlherr for helpful diskussions.
   
    A.B.\ has been supported by the DFG (German Research Foundation) – Project-ID 427320536 – SFB 1442, as well as under Germany’s Excellence Strategy EXC 2044 390685587, Mathematics  M\"unster: Dynamics \-- Geometry \-- Structure; W.L.\ under Germany's Excellence
   Strategy \--- GZ 2047/1, Projekt-ID 390685813, Hausdorff Center for Mathematics at Bonn; S.W.\ through the DFG Heisenberg project WI 4079/6".

\section{The flow space}
  
  Let $X$ be a locally compact ${\rm CAT}(0)$-space with an isometric proper cocompact smooth action of the td group $G$.
  In the presence of a cocompact isometric  action local compactness implies completeness and so 
  by Hopf-Rinow~\cite[Prop.~I.3.7]{Bridson-Haefliger(1999)} $X$ is a proper metric space, i.e.\ closed balls are compact.
  A generalized geodesic in $X$ is a continuous map $c \colon \IR \to X$ for which there exits an open interval $U \subseteq \IR$ such that $c|_U$ is an isometric embedding and $c|_{\IR \setminus U}$ is locally constant.
  Here $U$ may be unbounded, in particular $U = \IR$ is allowed.   
  The flow space $\FS = \FS(X)$ associated to $X$ is the space of all generalized geodesics $c \colon \IR \to X$.
  A metric on the flow space is defined by
  \begin{equation*}
  	d_{\FS} (c,c') = \int_\IR d_X(c(t),c'(t))2 e^{-|t|} dt.
  \end{equation*}  
  The flow on $\FS$ is the $\IR$-action defined by reparametrization: $\Phi_\tau(c)(t) = c(\tau+t)$.
  For $c \in \FS$ we define $K_c$ as the subgroup of all $g \in G$ with $gc(t)=c(t)$ for all $t$, $V_c$ as the subgroup of all $g \in G$ for which there is $\tau_g$ such that $gc(t)=c(t+\tau_g)$ for all $t$.
  For $g \in V_c \setminus K_c$, $\tau_g$ is unique
  and we define $\tau_c := \inf \{ |\tau_g| \mid  g \in V_c \setminus K_c \}$.
  If $V_c = K_c$, then $\tau_c = \infty$.
  If $\tau_c < \infty$ then we say that $c$ is \emph{periodic}.  
  We have $K_c \subseteq V_c$. 
  With these definitions~\cite[Assum.~2.7]{Bartels-Lueck(2023almost)} requires that there is $\FS_0 \subseteq \FS$ such that
  \refstepcounter{theorem}
  \begin{enumerate}[label=(\thetheorem\alph*),
                 align=parleft, 
                 leftmargin=*,
                 itemsep=1pt
                  ]
  \item\label{fund-domain} $G \cdot \FS_0 = \FS(X)$;
  \item\label{V} for $\ell > 0$ and $c_0 \in \FS_0$ there exists an open
    neighborhood $U$ of $c_0$ in $\FS_0$ such that for all $c \in U$ with
    $\tau_c \leq \ell$ we have $V_c \subseteq V_{c_0}$.
  \end{enumerate}
  
  Assume now that there is  a collection $\cala$ of closed convex subspaces of $X$ satisfying
  \refstepcounter{theorem}
  \begin{enumerate}[label=(\thetheorem\alph*),
                 align=parleft, 
                 leftmargin=*,
                 itemsep=1pt
                  ]
  	\item \label{nl:transitive} for all $A, A' \in \cala$ there is $g \in G$ with $g(A) = A'$ and $g|_{A \cap A'} = \id_{A \cap A'}$;
  	\item \label{nl:discret} for every $A \in \cala$ the quotient $\Gamma_A = \{g \in G \mid gA = A\}/\{g \in G \mid g|_A = \id_A\}$ is discrete;
    \item \label{nl:two-points} for any two points $x,y \in X$ there is $A \in \cala$ with $x,y \in A$. 
  \end{enumerate}
  Later $X$ will be a building and $\cala$ an apartment system.
  For $A \in \cala$ we let $\FS(A) \subseteq \FS(X)$ be the subspace of all generalized geodesics in $A$.

\begin{lemma}
	\label{lem:fund-domain}
	Let $A \in \cala$. 
	Then $G \cdot \FS(A) = \FS(X)$, i.e.~\ref{fund-domain} holds with $\FS_0 := \FS(A)$. 
\end{lemma}

\begin{proof}   
   We argue as in the case of a Euclidean building, see~\cite[Lemma~A.1]{Bartels-Lueck(2023almost)}:
   Let $c \in \FS(X)$ be given.
   By~\ref{nl:two-points} $c(0) \in A'$ for some $A' \in \cala$. 
   By~\ref{nl:transitive} the action of $G$ on $\cala$ is transitive.
   Thus we can assume $c(0) \in A$.
   By~\ref{nl:two-points} we find $A_n \in \cala$ that contains
   $c(\pm n)$.  
   By convexity $A_n$ then contains $c([-n,n])$.  
   Using~\ref{nl:transitive} we find $g_n \in G$ such that $g_n A = A_n$ and $g_n$ fixes $A \cap A_n$.  
   Then $(g_n)_{n \in \IN}$ is a sequence in the compact subgroup of $G$ that fixes $c(0)$ and has therefore an accumulation point $g$.  
   As the action of $G$ on $X$ is continuous and smooth, 
   there is for each $n$ a $k_n$ such that $g_{k_n} c(t) = g c(t)$ for all $t \in [-n,n]$. 
  It follows that $gA$ contains the image of $c$ and so $c \in g\cdot\FS(A)$.
\end{proof}

\begin{lemma}\label{lem:discrete-orbits-in-FS(A)} Let $c_0 \in \FS(A)$.  Then
  $Gc_0 \cap \FS(A)$ is discrete.
\end{lemma}

\begin{proof}
  Let $g \in G$ with $gc_0 \in \FS(A)$. 
  In particular, $c_0(\IR) \subseteq A \cap gA$.
  By~\ref{nl:transitive} there is $h \in G$ with $hgA=A$ and $h|_{A \cap gA} = \id_{A \cap gA}$. 
  Then $hgc_0=gc_0$ and $hg \in G_A$.
  Thus, by~\ref{nl:discret}, $gc_0 \in G_A c_0 = \Gamma_A c_0$. 
  Consequently, $Gc_0 \cap \FS(A) = \Gamma_A c_0$.
  
  By assumption proper $G \curvearrowright X$ is proper.
  By~\cite[Lemma~2.3]{Bartels-Lueck(2023almost)} this implies that $G \curvearrowright \FS(X)$ is proper as well.
  It follows that $G_A \curvearrowright \FS(A)$ is proper.
  Thus $G_A \curvearrowright \FS(A)$ is proper and has in particular discrete orbits.
  Therefore $Gc_0 \cap \FS(A) = \Gamma_A c_0$ is discrete. 
\end{proof}

\begin{lemma}\label{lem:about-tau_c} Let $c \in \FS$ be periodic.  Then
  there is $v \in V_c$ with $vc=\Phi_{\tau_c}(c)$ and any
  such $v$ together with $K_c$ generates $V_c$.
\end{lemma}

\begin{proof}
	This is~\cite[Lemma~2.8]{Bartels-Lueck(2023almost)}.
\end{proof}

\begin{lemma} \label{lem:V}
	Let $A \in \cala$. 
	Then~\ref{V} holds with $\FS_0 := \FS(A)$. 
\end{lemma}

\begin{proof}
   Let $\ell > 0$ and $c_0 \in \FS(A)$ be given.
   We first claim that there is $\epsilon > 0$ with the following property:
   if $g \in G$, $t \in [-\ell,\ell]$, $t \neq 0$, $c \in \FS(A)$ with 
   $d_{\FS}(c,c_0) < \epsilon$ and $gc = \Phi_t(c)$  then there is $h \in G$ with $hc=c$ and $hgc = \Phi_t(c)$.
   
   We argue by contradiction.
   Then there are $t_n \in [-\ell,\ell]$, $t_n \neq 0$, $g_n \in G$, $c_n \in \FS(A)$ with $g_n c_n = \Phi_t c_n$ and $c_n \to c_0$, but there is no $h \in G$ with $hc_n = c_n$ and $hg_nc = \Phi_{t_n}(c)$ for any $n$.
   The topology of $\FS(X)$ induced by its metric is the topology of uniform convergence on compact sets.
   In particular $d_X(c_n(0),c(0)) < \epsilon_n$ with $\epsilon_n \to 0$.
   We have $d_X(c(0),g_nc(0)) \leq d_X(c(0),c_n(0)) + d_X(c_n(0),g_nc_n(0)) + d_X(g_nc_n(0),g_nc(0)) \leq 2 \epsilon_n + d_X(c_n(0),c_n(t_n)) \leq 2 \epsilon_n \ell$.
   Thus the $g_n(c(0))$ vary over a bounded and therefore compact set in $X$.
   Since the action of $G$ on $X$ is proper this implies that the $g_n$ form a relative compact set in $G$.
   After a subsequence we can assume that $g_n \to g$ for some $g \in G$.
   After a further subsequence we can assume $t_n \to t$.
   Now $g_nc_0 \to  gc_0$.  
   As $G \curvearrowright \FS$ is isometric we also
   have $g_nc_n \to gc_0$.  
   Thus $gc_0 = \Phi_t(c_0)$.  
   Using~\ref{nl:transitive} we choose $h_n \in G$ with $A = h_n g_n A$ and $h_n|_{A \cap g_n A} = \id_{A \cap g_n A}$.
   We have $c_n = \FS(A)$ and $c_n = \Phi_{-t_n}(g_nc_n) \in \FS(g_nA)$.
   In particular $h_n c_n = c_n$ and $h_ng_nc_n=h_n\Phi_{t_n}(c_n)=\Phi_{t_n}(h_nc_n)=\Phi_{t_n}(c_n)$.
   Then  $h_n$ fixes $c_n(0)$ and $c_n(0) \to c(0)$.
   Thus the $h_n$ also form a relatively compact set in $G$ and after a subsequence we can assume $h_n \to h$.
   Now $h_ng_nc_0 \to hgc_0$. 
   As $h_ng_nA=A$ we have $h_ng_nc_0 \in \FS$ for all $n$. 
   By Lemma~\ref{lem:discrete-orbits-in-FS(A)} the $h_ng_n c_0$ form a discrete set in $\FS(A)$.
   As the $h_ng_n c_0$ also converge they must be eventually constant.
   So after a subsequence we can assume $h_ng_n c_0 = hgc_0$ for all $n$.
   Also $h_ng_nc_n=\Phi_{t_n}(c_n)$ implies $hgc_0=\Phi_{t}(c_0)$.
   Thus $h_ng_nc_0=\Phi_{t}(c_0)$ for all $n$.
   To arrive at the contradiction we aim for we need to show $t=t_n$.
   As $c_n$ is periodic (because $t_n \neq 0$) it is a bi-infinite geodesic.
   In particular, $d_X(c_n(Nt_n),c_n(0)) = Nt_n$ for all $N$.
   Thus $Nt_n = d_X(c_n(Nt_n),c_n(0)) \leq 2\epsilon_n + d(c_0(Nt),c_0(0)) \leq 2\epsilon_n + Nt$ for all $N$.
   Therefore $t \geq t_n$, and  $c$ is periodic as well. 
   Now conversely, $Nt = d_X(c(Nt),c(0)) \leq 2\epsilon_n + d(c_n(Nt_n),c_n(0)) \leq 2\epsilon_n + Nt_n$, implying $t_n \geq t$.
   Altogether, $t=t_n$ and so $h_ng_nc_0=\Phi_{t_n}(c_0)$ is the desired contradiction.
   This proves our first claim. 

   We observe that in the situation of the first claim the Flat Strip Theorem~\cite[p.182]{Bridson-Haefliger(1999)} applies.
   Thus for $c \in \FS(A)$ with $\tau_c \leq \ell$ 
   and $d_{\FS}(c,c_0) < \epsilon$ there is an isometric embedding $\IR \times [0,D] \to X$
   whose restriction to $\IR \times \{ 0 \}$ and $\IR \times \{D\}$ are (after reparametrization) $c$ and $c_0$.
   This implies that
   $c_0(\IR)$ is contained in the closure of the convex hull of $c_0(0) \cup c(\IR)$.
   In particular, if $h \in G$ fixes $c$ (in $\FS(X)$) and $c_0(0)$ (in $X$), then it fixes also $c_0$ (in $\FS(X)$).    

   Our second claim is that there is $\epsilon > 0$ with the following property: if $h \in G$, $c \in \FS(A)$ with
   $\tau_c \leq \ell$, $d_{\FS}(c,c_0) < \epsilon$, and $hc=c$, then $hc_0=c_0$ as well.
   
   We argue again by contradiction.
   Then there are $c_n \in \FS$, and $h_n \in G$ with $\tau_{c_n} \leq \ell$, $h_nc_n=c_n$, $c_n \to c_0$ but $h_nc_0\neq c_0$. 
   As $\tau_{c_n} \leq \ell$ there are $t_n \in [-\ell,\ell]$, $t_n \neq 0$ and $g_n \in G$ with $g_nc_n = \Phi_{t_n}(c_n)$. 
   Passing to a subsequence we can assume that the first claim applies to all $c_n$.
   Arguing as in the first claim we can pass to a further subsequence and assume that $h_n \to h$ for some $h \in G$.
   We have $c_n(0) \to c(0)$ and therefore $h$ fixes $c_0(0)$.
   As the action of $G$ on $X$ is smooth the stabilizer of $c_0(0)$ is open in $G$.
   Therefore, after passing to a subsequence, $h_n c_0(0) = c_0(0)$.
   Now the above observation using the Flat Strip Theorem implies that $h_n$ fixes $c_0$, contradicting our assumption.   
   This proves the second claim.   
   
   Choose now $\epsilon > 0$ such that the conclusion of both the first and the second claim hold.
   Let $c \in \FS(A)$ with $\tau_c \leq \ell$ and $d_{\FS}(c,c_0) < \epsilon$.
   The second claim implies directly that $K_c \subseteq K_{c_0} \subseteq V_{c_0}$.
   By Lemma~\ref{lem:about-tau_c} there is $v \in V_{c}$ with $vc=\Phi_{\tau_c}(c)$ such that $v$ together with $K_c$ generates $V_c$.
   So it remains to check $v \in V_{c_0}$.
   The first claim gives us $h \in K_c$ with $hv \in V_{c_0}$.
   As $K_c \subseteq V_{c_0}$ this implies $v \in V_{c_0}$.
\end{proof}

In summary we have the following.

\begin{proposition} \label{prop:FJ-axiomatic}
	Let $G$ be a td-group.
	Assume that $G$ admits an isometric proper cocompact smooth action  on a 
	locally compact ${\rm CAT}(0)$-space $X$ with a collection of closed convex subspaces $\cala$ satisfying~\ref{nl:transitive}, \ref{nl:discret}, and~\ref{nl:two-points}.
	
	Then $G$ satisfies the $\CVCYC$-Farrell--Jones conjecture. 	
\end{proposition}

\begin{proof}
	By Lemma~\ref{lem:fund-domain} and~\ref{lem:V} the action $G \curvearrowright X$ satisfies~\cite[Assum.~2.7]{Bartels-Lueck(2023almost)}.
	Thus~\cite[Thm.~1.2]{Bartels-Lueck(2023almost)} applies in this situation. 
	This is all the input on $G$ needed for the proof of the $\CVCYC$-Farrell--Jones conjecture for $G$ in~\cite[Thm.~5.15]{Bartels-Lueck(2023K-theory_red_p-adic_groups)}.    
\end{proof}

\section{Strongly transitive actions on buildings}
Let $X$ be the Davis--Moussong realization \cite{Davis08}\cite[Chapter~12]{AbramenkoBrown08} of a building: it is a complete  ${\rm CAT}(0)$ space. 
It is also a finite dimensional affine cell complex and has therefore finite covering dimension. 
There is a compact \emph{model chamber} $D$ on which $\operatorname{Aut}(X)$ acts through a finite quotient and an equivariant topological quotient map $\tau \colon X \to D$. 
The image $\tau(x)$ is the \emph{type} of $x$. 
Images of splittings of $\tau$ are (closed) \emph{chambers}.
An automorphisms of $X$ is said to be type-preserving if its action on $D$ is trivial.
The type-preserving automorphisms form a finite index subgroup of $\operatorname{Aut}(X)$.
Let $\cala$ be an \emph{apartment system} for $X$, meaning that
\begin{enumerate}[label=(B\arabic*), start=0,align=parleft, 
                 leftmargin=*,
                 itemsep=1pt
                  ]
  \item Every $A \in \cala$ is a subcomplex of $X$ that is isomorphic to the Davis--Moussong realization of a Coxeter complex.\label{item:bldg_cox}
  \item For $x,y \in X$ there is an $A \in \cala$ such that $x,y \in A$.\label{item:bldg_ex}
  \item For $A,A' \in \cala$ there is an isomorphism $\iota \colon A \to A'$ that restricts to the identity on $A \cap A'$.\label{item:bldg_uniq}
\end{enumerate}
Let $G$ be a topological group with an action on $X$.
The type preserving subgroup $G_0$ of $G$ consists of all $g$ that act type-preservingly on $X$.
We will require that $G_0$ acts \emph{strongly transitively} on $X$.
This means that $G_0$
 acts transitively on pairs $(C,A)$ where $A \in \cala$ and $C$ is a chamber of $A$. 
Equivalently, $G_{0}$ acts transitively on $\cala$ and $G_{0,A} = \{g \in G_0 \mid gA = A\}$ acts transitively on chambers of $A$. 
Since chambers are compact, this means in particular that $G_0$ (and thus $G$) act cocompactly on $X$. 



\begin{lemma}
  Under these assumptions the $A$ are closed convex subspaces satisfying \ref{nl:transitive}, \ref{nl:discret}, and~\ref{nl:two-points}.
\end{lemma}

\begin{proof}
   Observe that every apartment $A \in \cala$ is a geodesic metric space that is isometrically embedded in $X$ as a closed subspace by construction. 
   Since $X$, being ${\rm CAT}(0)$, is uniquely geodesic, it follows that $A$ is convex.

  Let $A,A' \in \cala$ and let $C$ be a chamber of $A$. By strong transitivity there exists a $g_1 \in G_0$ such that $g_1A = A'$. By \ref{item:bldg_uniq} there is an isomorphism $\iota \colon A' \to A$ restricting to the identity on $A \cap A'$. Let $C' = \iota(g_1C)$. Again by strong transitivity there is a $g_2 \in G_{0,A}$ with $g_2C' = C$, thus $\iota(g_1g_2C') = C'$. Now $\iota \circ (g_1g_2)|_A^{A'}$ is a type-preserving automorphism of $A$ that takes $C'$ to itself. It follows (see below) that it is the identity, i.e.\ $g_1g_2|_A^{A'} = \iota^{-1}$. Since $\iota$ restricts to the identity on $A \cap A'$, so does $g = g_1 g_2$.
  This shows \ref{nl:transitive}.
  
  The full group of type-preserving automorphisms of any apartment $A \in \cala$ is the underlying Coxeter group by construction. 
  In particular, an automorphism of $A$ that is type-preserving and fixes an interior point of a chamber, is trivial. 
  Write $G_{(A)} = \{g \in G \mid g|_A = \id_A\}$ for the pointwise stabilizer of $A$ and $G_{0,(A)}$ for its finite-index subgroup of type-preserving elements.
  Then $\Gamma_{0,A} = G_{0,A}/G_{0,(A)}$ continuously embeds into the Coxeter group and therefore is discrete.
  As $G_{0,A}$ has finite index in $G_A$, it follows that $G_{A}/G_{0,(A)}$ and hence also $\Gamma_A = G_{A}/G_{(A)}$ are discrete.  
  This shows \ref{nl:discret}.

  Finally, \ref{nl:two-points} is just \ref{item:bldg_ex}.
\end{proof}

We now discuss properness of the action.
On $X$ we consider the weak topology with respect to chambers. For ease of reference we record the following.

\begin{lemma}\label{lem:lf}
  The following notions are equivalent for $X$:
\begin{enumerate}[label=(\thetheorem\alph*),
                 align=parleft, 
                 leftmargin=*,
                 itemsep=1pt
                  ]
  \item $X$ is locally finite in the the polyhedral sense: every face is contained in finitely many faces;\label{item:lf_poly}
  \item $X$ is locally finite in the building theoretic sense: every panel is contained in finitely many chambers;\label{item:lf_bldg}
  \item $X$ is locally compact;\label{item:lc}
  \item $X$ is proper as a metric space.\label{item:proper}
\end{enumerate}
\end{lemma}

\begin{proof}
  In the Davis--Moussong realization, midpoints of chambers and panels are always realized, so if $X$ is locally finite, each panel is contained in finitely many chambers, showing that \ref{item:lf_poly} implies \ref{item:lf_bldg}. In a locally finite building, subbuildings of spherical type are finite, so the Davis--Moussong realization is locally compact, so \ref{item:lf_bldg} implies \ref{item:lc}. Since $X$ is complete, geodesic \ref{item:lc} implies \ref{item:proper} by Hopf--Rinow~\cite[Prop.~I.3.7]{Bridson-Haefliger(1999)}. Clearly, a polyhedral complex that is proper with respect to the weak topology has to be locally finite so \ref{item:proper} implies \ref{item:lf_poly}.
\end{proof}

Our present assumptions do not guarantee properness of either $X$ or the action. The following lemma takes care of this.

\begin{lemma}\label{lem:compact_B}
  Assume that the stabilizer $B$ in $G$ of a chamber is open and compact. Then $G$ and $X$ are locally compact and the action of $G$ on $X$ is continuous, smooth and proper.
\end{lemma}

See \cite{Kramer22} for a detailed discussion of related questions.

\begin{proof}
  As $G_0$ has finite index in $G$ it suffices to prove the statement for $G_0$ in place of $G$.

  That $G_0$ is locally compact is immediate.   
  The set of chambers can be identified with the space $G_0/B$, which is discrete since $B$ is open. Denote the chamber $1\cdot B/B$ by $C$ and let $P$ be a panel. Strong transitivity implies that $B$ acts transitively on the other chambers containing $P$. Since $B$ is compact and $G_0/B$ is discrete, it follows that there are finitely many of these. Hence by Lemma~\ref{lem:lf} $X$ is locally finite. Since $X$ is locally finite, on $\Aut(X)$ the permutation topology on the set $G_0/B$ of chambers and the compact-open topology for the weak topology on $X$ coincide so the action is continuous. Since $X$ is locally finite, all point-stabilizers are commensurable and since $B$ is compact all point-stabilizers are compact. Since $G_0$ preserves the polyhedral structure on $X$, it follows that the action is proper. Since the action is type-preserving, every cell-stabilizer contains a chamber-stabilizer, which is a conjugate of $B$ by strong transitivity and therefore open; it follows that the action is smooth.
\end{proof}

In summary we have.

\begin{proposition}\label{prop:geometric}
  Let $X$ be the Davis--Moussong realization of a building and let $\cala$ be a system of apartments. 
  Let a topological group $G$ act on $X$.
  Assume that its type-preserving subgroup $G_0$ acts
  strongly transitively on $X$ with respect to $\cala$ and assume that the stabilizer of a chamber is compact and open. 
  Then $X$ is a locally compact finite dimensional  ${\rm CAT}(0)$-space, the action of $G$ on $X$ is isometric, continuous, smooth, cocompact and proper. 
  Moreover, all $A \in \cala$ are 
  closed convex and the conditions~\ref{nl:transitive}, \ref{nl:discret}, and~\ref{nl:two-points} are satisfied.
\end{proposition}

\section{Buildings and groups} \label{sec:buildings-and-groups}

We now discuss what groups satisfy the conditions of Proposition~\ref{prop:geometric}. 
Assume that a group $G$ acts type-preservingly and strongly transitively 
on a building, let $B$ be the stabilizer of a chamber $C$ and $N$ be the setwise stabilizer of an apartment $\Sigma$ containing $B$ and assume that $G = \langle B,N \rangle$. 
Then $B \cap N$ is the pointwise stabilizer of $\Sigma$ and $W = N/(B \cap N)$ is the underlying Coxeter group; let $S$ be the fundamental reflections taking $C$ to its neighbors in $\Sigma$. 
Then $(G,B,N,S)$ is a \emph{Tits system} and $(B,N)$ a \emph{BN-pair} for $G$ \cite[Def.~B.24]{Marquis18}. 
Importantly, the converse is also true: if $(B,N)$ is a BN-pair for $G$, 
then $G$ acts strongly transitively on a building with chamber stabilizer $B$ and $N$ setwise stabilizing an apartment \cite[Prop.~B.26]{Marquis18}. 
Thus Proposition~\ref{prop:geometric} is equivalent to the following group-theoretic formulation:

\begin{proposition}\label{prop:group-theoretic}
  Let $G$ be a topological group that admits a BN-pair $(B,N)$ in which $B$ is compact and open. 
  Then $G$ admits a continuous, isometric, smooth, proper, cocompact action on a locally compact finite dimensional ${\rm CAT}(0)$ space with a collection $\cala$ of closed convex subspace satisfying the conditions~\ref{nl:transitive}, \ref{nl:discret}, and~\ref{nl:two-points}.
\end{proposition}

The following is minor generatization of Proposition~\ref{prop:group-theoretic}.

\begin{addendum} \label{add:group-theoretic}
  Let $G$ be a topological group with an open subgroup $G_0$ that admits a BN-pair $(B,N)$ in which $B$ is compact and open. Assume that $G = G_0 \cdot N_G(B)$.
  Then the conclusion of Proposition~\ref{prop:group-theoretic} remains true for $G$.
\end{addendum}

\begin{proof}
  That $G_0$ admits a BN-pair means that it acts strongly transitively on a thick building. Its parabolic subgroups are the subgroups of $G_0$ that contain a conjugate of $B$ and they are self-normalizing \cite[Theorem~6.43]{AbramenkoBrown08}. Since the building is locally finite by Lemmas~\ref{lem:lf} and~\ref{lem:compact_B} the Davis realization $X$ is the realization of the poset of subgroups of $G_0$ that are finite index overgroups of a conjugate of $B$. Thus the action of $G_0$ on $X$ can be identified with the conjugation action of $G_0$ on the realization of the poset just described. The condition $G = G_0 \cdot N_G(B)$ asserts that $G$ acts on this poset by conjugation.
\end{proof}

It remains to discuss which groups admit such a BN-pair. There are two main sources for such groups: Bruhat--Tits theory gives groups for which our results are known while Kac--Moody theory leads to new results.

The point of departure of both theories are reductive groups over fields. If $\bfG$ is a connected reductive group defined over a field $k$ there is an associated spherical \emph{Tits building} $\Delta_{\bfG,k}$ on which $\bfG(k)$ acts strongly transitively.

\subsection*{Bruhat--Tits}
If $k$ is a non-Archimedean local field with ring of integers $\calo$ then there is additionally an associated locally finite, euclidean \emph{Bruhat--Tits building} $X_{\bfG,k}$ on which $\bfG(k)$ acts and it contains an open subgroup $\bfG(k)^0$ that acts strongly transitively. The building $\Delta_{\bfG,k}$ is naturally identified with the visual boundary of $X_{\bfG,k}$, so one may regard $X_{\bfG,k}$ as a refinement of $\Delta_{\bfG,k}$. For every chamber $c$ of $X_{\bfG,k}$ there is an $\calo$-scheme structure $\mathfrak{G}$ on $\bfG$ such that $\mathfrak{G}(\calo)$ is (commensurable to) the stabilizer of $c$ in $\bfG(k)$ \cite[Section~3.4]{Tits79}, so in particular the stabilizer $B$ of $c$ is compact open.  Thus $\bfG(k)$ satisfies the hypotheses of Addendum~\ref{add:group-theoretic}. 
These examples are covered by~\cite{Bartels-Lueck(2023almost)}.

\subsection*{Kac--Moody}
In a different direction, Kac--Moody theory extends rather than refines $\Delta_{\bfG,k}$: If $\bfG$ is a Tits functor (of which a split semisimple group\footnote{Tits extends Chevalley--Demazure group schemes which explains the restriction to split semisimple groups. Recall that a reductive group is an almost-direct product of a semisimple group and a central torus, so the restriction to semisimple groups is harmless. The restriction to split groups is substantial, though note that quasi-split groups are considered in \cite{Remy02}} is a special case) and $k$ is a field, there is an associated \emph{twin building} $(\Delta_+,\Delta_-)$ on which $\bfG(k)$ acts strongly transitively. We discuss this following the standard references \cite{Remy02,AbramenkoBrown08,Marquis18}. An important feature that spherical buildings have is an \emph{opposition relation}: two chambers $C$ and $D$ may be opposite in some apartment $A$ which turns out to be unique. While in non-spherical apartments, chambers cannot be opposite, \emph{twin buildings} generalize spherical buildings to arbitrary types. They consist of two buildings $X_+$ and $X_-$ and an opposition relation between them. A \emph{twin apartment} consists of an apartment $A_+$ of $X_+$ and an apartment $A_-$ of $X_-$ such that opposition establishes a bijective correspondence between the chambers of $A_+$ and of $A_-$. Again, a pair of opposite chambers is contained in a unique twin apartment. Thus a strongly transitive action may equivalently be defined as an action that is transitive on pairs of opposite chambers. With these notions in place, the basic theory of twin buildings extends that of spherical buildings with minor modifications (the structure theory is richer, of course).

Just as strongly transitive actions of groups on buildings correspond to BN-pairs (as above), strongly  transitive actions of groups on twin buildings correspond to \emph{twin BN-pairs}: subgroups $B_+$, $B_-$, and $N$ which end up being the stabilizers of opposite chambers of $X_+$ and of $X_-$ and the setwise stabilizer of the twin apartment containing the two.

If $\bfG$ is a \emph{Tits functor} \cite[Def~7.47]{Marquis18} and $k$ is a field, the group $G = \bfG(k)$ admits a twin-BN-pair \cite[Thm.~7.69, Cor.~7.70]{Marquis18} and thus acts strongly transitively on a twin building $(X_+,X_-)$. The buildings $X_+$ and $X_-$ are locally finite (in the polyhedral sense when taking the Davis--Moussong realization) if $k$ is finite. The group $\bfG(k)$ is then countable discrete. However, (unless $\bfG$ is of spherical type so that $X_+$ is finite and $G$ is a finite group of Lie type) $\Aut(X_+)$ is non-discrete and so one may consider the closure of $\bfG(k)$ in $\Aut(X_+)$ to obtain a non-discrete locally compact group. In fact, besides this geometric completion there are various alternative completions that all lead to groups satisfying the hypotheses of Proposition~\ref{prop:group-theoretic}:

%

\begin{proposition} \label{prop:kac-moody_new}
  Let $\bfG$ be a Tits functor, let $k$ be a finite field, and let $G = \bfG(k)$ be the resulting Kac--Moody group.
  Let $\bar{G}$ be either the geometric, representation theoretic, the algebraic, or the scheme theoretic completion of $G$.
  
  Then $\bar{G}$ admits a  BN-pair with $B$ compact open. In particular $\bar{G}$ satisfies the
  satisfies the $\CVCYC$-Farrell--Jones conjecture.
\end{proposition}
\begin{proof}
  The group $\bar{G}$ admits a BN-pair with $B$ compact open.
  This is shown in \cite{Marquis18} in Theorem~8.16 and Proposition~8.17 for the geometric completion, in Theorem~8.23 and Proposition~8.24 for the representation theoretic completion, in Theorem~8.33 and Proposition~8.34 for the algebraic completion, and in Theorem~8.84 and Proposition~8.85 for the scheme theoretic completion.
  Finally apply Theorem~\ref{thm:FJ-for-BN}.
\end{proof}


\end{document}